\RequirePackage{etex}
\documentclass[11pt, a4paper]{article}
\usepackage[utf8]{inputenc}
\usepackage{inputenc}
\usepackage{amsmath}
\usepackage{amssymb}
\usepackage{authblk}
\usepackage{graphicx} 
\usepackage{tikz}
\usepackage{amsthm}
\usepackage[margin=1cm]{caption}
\usepackage{enumitem}
\usepackage{hyperref}
\usepackage{todonotes}
\usepackage{array,diagbox}

\usepackage[top=2.2cm, bottom=3.2cm, left=2.5cm, right=2.5cm]{geometry}
\newcommand{\A}{\mathcal{A}}
\newcommand{\SL}{\mathcal{SL}}
\newcommand{\E}{\mathcal{E}}
\newcommand{\dist}{\textup{dist}}

\newtheoremstyle{the}{13pt}{12pt}{\it}{}{\bfseries}{.}{ }{}
\newtheoremstyle{def}{13pt}{12pt}{}{}{\bfseries}{.}{ }{}
\newtheoremstyle{rem}{13pt}{12pt}{}{}{\it}{.}{ }{}
\newtheoremstyle{cla}{13pt}{12pt}{\it}{}{$\quad$}{.}{ }{}

{\theoremstyle{the} \newtheorem{theorem}{Theorem}[section]}
{\theoremstyle{the} \newtheorem{lemma}[theorem]{Lemma}}
{\theoremstyle{the} \newtheorem{proposition}[theorem]{Proposition}}
{\theoremstyle{the} \newtheorem{corollary}[theorem]{Corollary}}
{\theoremstyle{the} \newtheorem*{conjecture}{Conjecture}}
{\theoremstyle{the} \newtheorem{observation}[theorem]{Observation}}
{\theoremstyle{def} \newtheorem{definition}[theorem]{Definition}}
{\theoremstyle{def} }
{\theoremstyle{rem} \newtheorem*{remark}{Remark}}
{\theoremstyle{rem} }
{\theoremstyle{cla} \newtheorem{claim}{Claim}}

\newenvironment{proofclaim}[1][Proof]{\begin{proof}[#1]}{\end{proof}}

\usetikzlibrary{external,shapes}
\usetikzlibrary{decorations.pathreplacing}
\usetikzlibrary{positioning}
\usetikzlibrary{calc}
\usetikzlibrary{arrows.meta}
\usetikzlibrary{shapes.geometric}

\tikzstyle{noeud}=[circle,inner sep=1, minimum size =1 pt, line width = 1pt, draw=black, fill=black]
\newcommand{\convexpath}[2]{
[   
    create hullnodes/.code={
        \global\edef\namelist{#1}
        \foreach [count=\counter] \nodename in \namelist {
            \global\edef\numberofnodes{\counter}
            \node at (\nodename) [draw=none,name=hullnode\counter] {};
        }
        \node at (hullnode\numberofnodes) [name=hullnode0,draw=none] {};
        \pgfmathtruncatemacro\lastnumber{\numberofnodes+1}
        \node at (hullnode1) [name=hullnode\lastnumber,draw=none] {};
    },
    create hullnodes
]
($(hullnode1)!#2!-90:(hullnode0)$)
\foreach [
    evaluate=\currentnode as \previousnode using \currentnode-1,
    evaluate=\currentnode as \nextnode using \currentnode+1
    ] \currentnode in {1,...,\numberofnodes} {
-- ($(hullnode\currentnode)!#2!-90:(hullnode\previousnode)$)
  let \p1 = ($(hullnode\currentnode)!#2!-90:(hullnode\previousnode) - (hullnode\currentnode)$),
    \n1 = {atan2(\y1,\x1)},
    \p2 = ($(hullnode\currentnode)!#2!90:(hullnode\nextnode) - (hullnode\currentnode)$),
    \n2 = {atan2(\y2,\x2)},
    \n{delta} = {-Mod(\n1-\n2,360)}
  in 
    {arc [start angle=\n1, delta angle=\n{delta}, radius=#2]}
}
-- cycle
}

\title{The Avoider-Enforcer game on hypergraphs of rank 3\thanks{This research was partly supported by the ANR project P-GASE (ANR-21-CE48-0001-01)}}
\author[1]{Florian Galliot}
\author[2]{Valentin Gledel}
\author[3]{Aline Parreau}
\date{March 2025}

\affil[1]{Aix-Marseille Université, CNRS, Centrale Marseille, I2M, UMR 7373, 13453 Marseille, France}
\affil[2]{Université Savoie Mont Blanc, CNRS UMR5127, LAMA, Chambéry, F-73000, France}
\affil[3]{Univ Lyon, CNRS, INSA Lyon, UCBL, Centrale Lyon, Univ Lyon 2, LIRIS, UMR5205, F-69622 Villeurbanne, France.}

\begin{document}

\maketitle

\begin{abstract}
    In the Avoider-Enforcer convention of positional games, two players—Avoider and Enforcer—take turns selecting vertices from a hypergraph $H$. Enforcer wins if, by the time all vertices of $H$ have been selected, Avoider has completely filled an edge of $H$ with her vertices; otherwise, Avoider wins. In this paper, we first give some general results, in particular regarding the outcome of the game and disjoint unions of hypergraphs. We then determine which player has a winning strategy for all hypergraphs of rank~2, and for linear hypergraphs of rank~3 when Avoider plays the last move. The structural characterisations we obtain yield polynomial-time algorithms.
\end{abstract}

\section{Introduction}\strut
\indent Positional games have been introduced by Hales and Jewett \cite{Hales1963} and popularised by Erd\H{o}s and Selfridge~\cite{erdos}. The game board is a hypergraph $H$ with vertex set $V(H)$ and edge set $E(H)$. Two players take turns picking a previously unpicked vertex of the hypergraph. 
There are four main {\em conventions} for determining the winner.
In {\em strong games}, both players have the same objective: in the {\em Maker-Maker} convention, whoever fills an edge (i.e. picks all its vertices) wins the game, whereas in the {\em Avoider-Avoider} convention, whoever first fills an edge loses the game. In {\em weak games}, the players' objectives are complementary: in the {\em Maker-Breaker} convention, Maker aims at filling an edge while Breaker aims at preventing her from doing so, whereas in the {\em Avoider-Enforcer} convention, Avoider aims at avoiding filling an edge while Enforcer aims at forcing her to do so. The Maker-Maker and Maker-Breaker conventions characterise {\em achievement games}, whereas the Avoider-Avoider and Avoider-Enforcer conventions characterise {\em avoidance games}. Since all these conventions define finite two-player games with perfect information, either there exists a winning strategy for one of the players, or optimal play leads to a draw (this can only happen in strong games).

The first convention to be introduced was the Maker-Maker convention \cite{Hales1963}, which is the most natural one since it models the famous game of Tic-Tac-Toe. It can be noted, using a strategy-stealing argument \cite{Hales1963}, that there are no situations where the players would rather pass their turn, and therefore either the first player has a winning strategy or optimal play leads to a draw. Apart from that, few general results are known, as this convention is hard to tackle. There is no monotonicity (adding edges may harm the first player), and it is not known how to compute a winning strategy for the first player on a disjoint union of two games given a winning strategy on each of them. These difficulties are mainly due to the fact that the players pursue two goals at once: filling an edge, and preventing the opponent to fill one first.

This is the reason why the Maker-Breaker convention was introduced by Chv\'atal and Erd\H{o}s  \cite{Chvatal}. Constructive results are generally easier to achieve compared to the Maker-Maker convention: for instance, it is very easy to design optimal strategies on a disjoint union of two games from optimal strategies on each separate game. These games are monotonous, since adding an edge can clearly only benefit Maker. Some other general results are known, such as the ``Erd\"os-Selfridge criterion" or ``pairing strategies" which give sufficient conditions for Breaker to have a winning strategy. The Maker-Breaker convention is actually by far the most studied (see for example the books \cite{beck,positionalgames}). 
Maker-Breaker games on hypergraphs of small rank have been much investigated. Hypergraphs of rank at most 2 are trivially solved. Hypergraphs of rank 3 were first studied in the linear case (i.e. when any two edges share at most one vertex) by Kutz, who obtained a polynomial-time algorithm to decide which player has a winning strategy \cite{kutz}. This result has been recently extended to all hypergraphs of rank 3 by Galliot, Gravier and Sivignon \cite{MBrank3}. On the other hand, it has been known since 1978 that deciding the winner of a Maker-Breaker game is PSPACE-complete even if the edges have size 11 \cite{schaefer}. This result has since then been improved to edges of size~6~\cite{Rahman20206UniformMG} and very recently to edges of size~5~\cite{MBrank5}.

The other two conventions, which revolve around avoiding filling an edge, have received less attention but are now being increasingly studied \cite{HgameAE,gledeloijid,milos2,milos1}.
The way in which avoidance games are played is quite different. Indeed, in these conventions, players would often rather pass than pick a vertex. As a consequence, there are cases where the second player has a winning stragegy. All in all, methods used for achievement games mostly do not transfer to avoidance games. As an illustration, it took until 2023 to prove that deciding which player has a winning strategy is PSPACE-complete for the Avoider-Enforcer convention, over 40 years later than for the Maker-Breaker convention.
Among the two avoidance conventions, Avoider-Avoider games, being strong games, present a similar predicament as Maker-Maker games. Deciding which player (if any) has a winning strategy on a hypergraph of rank 2 is already PSPACE-complete in this convention \cite{colpspace2}.

This paper focuses on the Avoider-Enforcer convention, which owes its first general formulation to Lu \cite{Lu}. Since this is a weak convention i.e. the two players have complementary goals, one can hope that these games satisfy some convenient general properties, similar to those of Maker-Breaker games (despite the differences in the way these games are played).
For example, it is already known that the Erd\"os-Selfridge criterion can be adapted to the Avoider-Enfocer convention~\cite{AEgames}. Furthermore, pairing strategies have been used to prove the PSPACE-completeness of Avoider-Enforcer games \cite[Lemma 6]{gledeloijid}. That last paper illustrates how complexity proofs for Maker-Breaker games can often be adapted to Avoider- Enforcer games.
Beside these works, the Avoider-Enforcer convention seems to have been studied mostly for specific games \cite{FastAE,HgameAE,acyclicAE,starAE} and not for general hypergraphs, except those with pairwise-disjoint edges which are entirely solved \cite{AEdisjoint}.

\

{\bf Our contribution.} The aim of this paper is to give general properties of Avoider-Enforcer games and study this convention in hypergraphs of small rank.
An important property that we establish is that, in contrast with the Maker-Breaker convention, the relevant distinction to make when studying Avoider-Enforcer games is based on who plays the last move of the game (when all vertices are picked) rather than who plays the first move. Indeed, any player that has a winning strategy playing last also has a winning strategy playing second-to-last. This result leads us to define the {\em outcome} of a hypergraph as the identity of the player that has a winning strategy in each case (Avoider playing last or Enforcer playing last). We then show that the outcome of a disjoint union of games is uniquely determined by the outcomes of its components.
We then consider hypergraphs of small rank. For hypergraphs of rank 2, we give a structural characterisation of the outcome. As for hypergraphs of rank 3, we initiate their study by investigating the linear case, which was the first case to be solved in the Maker-Breaker convention \cite{kutz}. We completely solve the subcase where Avoider plays last, by providing a structural charaterisation of hypergraphs on which Avoider has a winning strategy playing last. All our structural characterisations can straightforwardly be recognised in polynomial time.
Definitions and general results are given in Section \ref{sec:general}. Section \ref{sec:rank2} is dedicated to hypergraphs of rank 2, and Section \ref{section_rank3} to linear hypergraphs of rank 3 when Avoider is the last player.

\section{General results on the Avoider-Enforcer game}\label{sec:general}
\subsection{Definitions and notations}

\paragraph{Hypergraphs.} Let $H$ be a hypergraph. We denote its vertex set by $V(H)$ and its edge set by $E(H)\subseteq 2^{V(H)}$. We say that $H$ is {\em even} (resp. {\em odd}) if it has an even (resp. odd) number of vertices. A {\em neighbour} of a vertex $x$ in $H$ is a vertex $y \neq x$ such that some edge of $H$ contains both $x$ and $y$. The {\em degree} of a vertex in $H$ is the number of edges it belongs to. A vertex is called {\em isolated} if it has degree 0. A {\em leaf-edge} of $H$ is an edge with at most one vertex of degree greater than 1. A {\em subhypergraph} of $H$ is a hypergraph $H'$ such that $V(H') \subseteq V(H)$ and $E(H') \subseteq E(H)$. A {\em transversal} of $H$ is a set of vertices that intersects every edge of $H$. \newline
Let $k \in \mathbb{N}$. An edge is a {\em $k$-edge} if it contains exactly $k$ vertices.
A hypergraph has {\em rank $k$} if all its edges have at most $k$ vertices, and it is {\em $k$-uniform} if every edge contains exactly $k$ vertices.  \newline
A hypergraph is {\em linear} if any two distinct edges share at most one vertex. A {\em linear tree} (following the definition from \cite{lineartree}) is recursively defined as a hypergraph that either is an isolated vertex or can be obtained from a linear tree $T$ by adding an edge that shares exactly one vertex with $T$. \newline
A {\em walk} in a hypergraph $H$ is a sequence $(e_1,\ldots,e_{\ell})$ of edges of $H$ such that $e_i \cap e_{i+1} \neq \varnothing$ for all $1 \leq i \leq \ell-1$. We say a walk $(e_1,\ldots,e_{\ell})$ is {\em simple} if $e_i \cap e_j = \varnothing$ for all $1 \leq i,j \leq \ell$ such that $|i-j|\geq 2$. 
We say a hypergraph $H$ is {\em connected} if, for all distinct $a,b \in V(H)$, there exists a walk $(e_1,\ldots,e_{\ell})$ in $H$ such that $a \in e_1$ and $b \in e_{\ell}$. The {\em connected components} of $H$ are the maximal connected subhypergraphs of $H$. If $H$ is connected, then a vertex $y \in V(H)$ is called a {\em cut vertex} of $H$ if $H^{-y}$ is disconnected or a {\em non-cut vertex} otherwise. \newline
The {\em union} of two hypergraphs $H_1$ and $H_2$, denoted by $H_1 \cup H_2$, is the hypergraph with vertex set $V(H_1)\cup V(H_2)$ and edge set $E(H_1)\cup E(H_2)$. This is a {\em disjoint union} if $V(H_1)$ and $V(H_2)$ are disjoint. \newline
For further definitions and notations on hypergraphs, we refer the reader to the reference books \cite{berge} and \cite{bookhypergraph}.
\paragraph{Avoider-Enforcer Games.} Let $H$ be a hypergraph. An Avoider-Enforcer game on $H$ is played by two players, Avoider and Enforcer. In turn, the players pick a vertex of $H$ that has not been picked yet. The game ends when all the vertices have been picked. Enforcer has won if Avoider has {\em filled} an edge of $H$ (i.e. Avoider has picked all the vertices of some edge), otherwise Avoider has won. By convention, if $H$ contains the empty edge, then Enforcer wins.
A {\em round} is a sequence of two moves, one for the first player and one for the second player.
We say that {\em Avoider wins on $H$ as first} (resp. {\em second}, resp. {\em last}, resp. {\em second-to-last}) {\em player} if Avoider has a winning strategy for the Avoider-Enforcer game played on $H$ when Avoider is the first (resp. second, resp. last, resp. second-to-last) player. We introduce analogous definitions for Enforcer. Being the last player means playing the last move of the game: in other words, the last player is the first player if $H$ is odd or the second player if $H$ is even.

\noindent It can be noticed that the opposite of Avoider filling an edge is Enforcer picking at least one vertex in each edge i.e. filling a transversal. Therefore, Avoider is an ``enforcer of transversals", while Enforcer is an ``avoider of transversals". This role-reversal argument yields the following observation.

\begin{observation}\label{obs_dual}
	Let $H$ be a hypergraph, and denote by $H^{T}$ the hypergraph with same vertex set as $H$ whose edges are the transversals of $H$. Avoider wins on $H$ as last player (resp. second-to-last player) if and only if Enforcer wins on $H^T$ as last player (resp. second-to-last player).
\end{observation}

\noindent A convenient property of weak positional games is that there is never a need to consider hypergraphs where vertices have already been picked, as this case reduces to the case where the game has not started yet. Indeed, let us introduce the following notations, where $H$ is any hypergraph:
\begin{itemize}[noitemsep,nolistsep]
	\item For all $x \in V(H)$, we denote by $H^{+x}$ the hypergraph defined by $V(H^{+x})=V(H)\setminus\{x\}$ and $E(H^{+x})=\{e \setminus \{x\} \mid e \in E(H)\}$.
	\item For all $y \in V(H)$, we denote by $H^{-y}$ the hypergraph defined by $V(H^{-y})=V(H)\setminus\{y\}$ and $E(H^{-y})=\{e \in E(H) \mid y \not\in e\}$.
\end{itemize}
\noindent The operators ${}^{+x}$ and ${}^{-y}$ can be used to update the hypergraph during the game, so that the edges of the updated hypergraph represent what remains of the initial edges. Whenever Enforcer picks a vertex $y$, any edge containing $y$ is ``killed" since it will never be filled by Avoider. Using these notations, after one round of play of the Avoider-Enforcer game played on $H$ with Avoider picking $x$ and Enforcer picking $y$, it is as if a new game starts on the updated hypergraph $H^{+x-y}$. Therefore, all results on hypergraph classes which are stable under the operators ${}^{+x}$ and ${}^{-y}$ also hold for hypergraphs on which some moves have already been made, since we can always reduce to the case of a ``fresh" hypergraph using the above updating method. The winner of the game can be defined by induction. For example, when Enforcer is the first player: if $|V(H)|\leq 1$ then Enforcer wins on $H$ if and only if $\varnothing \in E(H)$, otherwise Enforcer wins on $H$ if and only if there exists $y \in V(H)$ (Enforcer's pick) such that for all $x \in V(H)\setminus\{y\}$ (Avoider's answer) Enforcer wins on $H^{-y+x}$.

\subsection{Outcome and importance of the last move}\label{sec:lastmove}

\noindent In Maker-Breaker games, it is well known (see for example \cite{positionalgames}) that any player who has a winning strategy as second player also has a winning strategy as first player. This leads to only three possible outcomes: Maker wins as first or second player, Breaker wins as first or second player, or the first player wins. On the contrary, in Avoider-Enforcer games, winning as second player is not a guarantee that one can win as first player, nor the other way around. As illustrated in Figure \ref{fig:second_player_wins}, all four outcomes exist when it comes to who plays first.

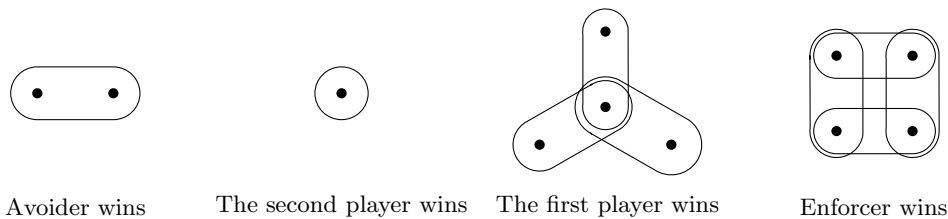
\begin{figure}[ht]
    \centering
        \begin{tikzpicture}
        
        \node at (0,0){
            \begin{tikzpicture}
            	\node[noeud] (a) at (-1,0){}; 
				\node[noeud] (b) at (0,-0){};				
				\draw \convexpath{a,b}{0.35cm};
            \end{tikzpicture}
        };
        
        \node at (3.5,0){
            \begin{tikzpicture}
				\node[noeud] (a) at (0,0){};			
				\draw (a) circle (0.35cm);
            \end{tikzpicture}
        };        
        
		\node at (7,0){
			\begin{tikzpicture}
				\node[noeud] (o) at (0,0){};
				\node[noeud] (a) at (90:1){};
				\node[noeud] (b) at (210:1){};
				\node[noeud] (c) at (330:1){};				
				\draw \convexpath{o,a}{0.3cm};				
				\draw \convexpath{o,b}{0.35cm};				
				\draw \convexpath{o,c}{0.4cm};
			\end{tikzpicture}
		};  
        
        \node at (10.5,0){
            \begin{tikzpicture}
				\node[noeud] (a) at (0,0){};
				\node[noeud] (b) at (1,0){};
				\node[noeud] (c) at (1,1){};
				\node[noeud] (d) at (0,1){};				
				\draw \convexpath{a,b}{0.3cm};				
				\draw \convexpath{b,c}{0.35cm};				
				\draw \convexpath{c,d}{0.3cm};				
				\draw \convexpath{d,a}{0.35cm};
            \end{tikzpicture}
        };        

        \node at (0,-1.5) {\footnotesize Avoider wins};
        \node at (3.5,-1.5) {\footnotesize The second player wins};
        \node at (7,-1.5) {\footnotesize The first player wins};
        \node at (10.5,-1.5) {\footnotesize Enforcer wins};
        
    \end{tikzpicture}
    \caption{Examples of hypergraphs for each possible outcome.}
    \label{fig:second_player_wins}
\end{figure}

\noindent However, a result similar to that of Maker-Breaker games can still be obtained, by considering the last move rather than the first move:

\begin{theorem}\label{theo:last}
Let $H$ be a hypergraph. If Avoider (resp. Enforcer) wins on $H$ as last player, then Avoider (resp. Enforcer) also wins on $H$ as second-to-last player.
\end{theorem}

\begin{proof}

By Observation \ref{obs_dual}, it suffices to consider Enforcer's side of things. Suppose Enforcer wins on $H$ as last player. We will use strategy stealing to show that Enforcer can basically apply the same strategy when he is the second-to-last player. The proof will distinguish between two cases depending on the parity of $H$.

First assume that $H$ is even, meaning Enforcer has a winning strategy $\Sigma$ as second (and last) player.
When playing first, Enforcer can pick an arbitrary vertex $u$, ``forget" this move, and then act as second player applying the strategy $\Sigma$. In other words, Enforcer imagines a ``fictitious game" where he is the second player and $u$ has not been played, and where he applies $\Sigma$. Enforcer copies each Avoider move into the fictitious game, and then copies the answer given by $\Sigma$ in the real game. The only exception is whenever $\Sigma$ requires that Enforcer picks the forgotten vertex, in which case he picks another arbitrary vertex instead in the real game, which becomes the new forgotten vertex. The fact that $\Sigma$ is a winning strategy for Enforcer as second player ensures that Avoider fills an edge at some point in the fictitious game. Since the set of vertices picked by Avoider is the same in both games, Avoider also fills an edge in the real game, which concludes.

Assume now that $H$ is odd, meaning Enforcer has a winning strategy $\Sigma$ as first (and last) player. When playing second, Enforcer will ``forget" the vertex $u$ that Avoider picked as her first move, and then act as first player applying the strategy $\Sigma$. Same as before, whenever $\Sigma$ requires that Enforcer picks the forgotten vertex, he picks another arbitrary vertex instead in the real game, which becomes the new forgotten vertex. The fact that $\Sigma$ is a winning strategy for Enforcer as first player ensures that Avoider fills an edge at some point in the fictitious game. Since the set of vertices picked by Avoider in the real game is a superset of what it is in the fictitious game (it is the same plus the vertex $u$), Avoider also fills an edge in the real game, which concludes.
\end{proof}

\noindent Theorem~\ref{theo:last} invites us to consider outcomes in terms of the player playing last, rather than the player playing first as is done for achievement games. In the following, we will denote by $o(H)$ the outcome of the game played on a hypergraph $H$.
Since any player who has a winning strategy as last player also has a winning strategy as second-to-last player, there are only three possible values for $o(H)$: either Avoider wins in all cases (denoted by $o(H)=\A$), or Enforcer wins in all cases (denoted by $o(H)=\E$), or the player playing second-to-last wins (denoted by $o(H)=\SL$).
Note that one can recover the traditional outcomes by considering the parity of $H$: if $o(H)=\SL$, then the first player has a winning strategy if $H$ is even, and the second player has a winning strategy if $H$ is odd.

\noindent In the Maker-Breaker convention, it is possible to reduce to the case where, say, Maker starts. Indeed, if Breaker starts then it suffices to consider all possibilities for Breaker's first move $y$, after which a new game begins on the updated hypergraph $H^{-y}$ with Maker as first player. Similarly, in the Avoider-Enforcer convention, where the key information is who plays last rather than who plays first, we would like to be able to choose the last player and always reduce to that case. For this, an idea would be to consider all possibilities, not for the first move (chosen by the first player), but for the last move (forced by the second-to-last player). Indeed, the second-to-last player has the luxury of avoiding any particular vertex $x$ of their choice: by never picking $x$, they will force the last player to eventually pick $x$. Therefore, an intuition for a reduction would be the following: if $x$ is the ``worst" vertex to pick, then the second-to-last player could just as well give $x$ to their opponent before the game starts, and then play as last player on what remains. One direction of this reasoning does hold:

\begin{proposition}\label{prop:lastmove}
	Let $H$ be a hypergraph.
	\begin{itemize}[noitemsep,nolistsep]
		\item If there exists $y \in V(H)$ such that Avoider wins on $H^{-y}$ as last player, then Avoider wins on $H$ as second-to-last player.
		\item If there exists $x \in V(H)$ such that Enforcer wins on $H^{+x}$ as last player, then Enforcer wins on $H$ as second-to-last player.
	\end{itemize}
\end{proposition}

\begin{proof}
	We only prove the first assertion, as the second is proved analogously. We show the contrapositive. Suppose that Enforcer has a winning strategy $\Sigma$ on $H$ as last player. Let $y \in V(H)$: we want to show that Enforcer wins on $H^{-y}$ as second-to-last player. Note that the first player is the same in both games, since the last player and the parity are different. Initially, Enforcer follows the strategy $\Sigma$ in the fictitious game played on $H$, and copies the moves in the real game played on $H^{-y}$. In the fictitious game, being the last player, Enforcer will inevitably pick $y$ at some point. At that moment, the updated hypergraph $H'$ is the same in both games. By optimality of $\Sigma$, Enforcer wins on $H'$ as last player, therefore Enforcer also wins on $H'$ as second-to-last player by Theorem \ref{theo:last}. In conclusion, Enforcer wins the real game as a whole.
\end{proof}

\noindent Unfortunately, neither converse holds in Proposition \ref{prop:lastmove} (see Figure \ref{fig:counterexamples} for counterexamples). 
This is not surprising: the ``worst vertex" might change depending on what happens during the game, so it might not be a good idea for the second-to-last player to identify one permanently before the game begins.

\begin{figure}[h]
	\centering
	\includegraphics[scale=.6]{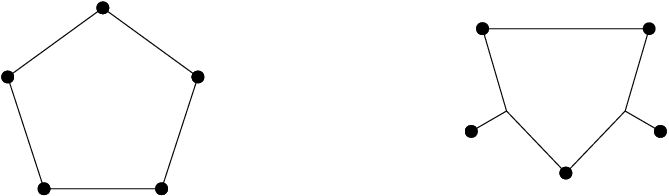}
	\caption{Counterexamples to the converse of the first (left) and the second (right) assertions of Proposition \ref{prop:lastmove}. As in all figures from now, we represent edges of size 2 with a straight line and edges of size 3 with a ``claw" shape. }\label{fig:counterexamples}
\end{figure}

\noindent All in all, we do not know of an easy way to reduce to the case where Avoider or Enforcer is the last player, so we will have to consider both possibilities in our studies. However, we do have a simplification when a 1-edge is involved.

\begin{proposition}\label{prop:1-edge}
	Let $H$ be a hypergraph containing a 1-edge $\{y\}$. Then:
	\begin{itemize}[noitemsep,nolistsep]
		\item Avoider loses on $H$ as last player.
		\item Avoider wins on $H$ as second-to-last player if and only if Avoider wins on $H^{-y}$ as last player.
	\end{itemize}
\end{proposition}

\begin{proof}
	The first assertion is obvious: since Enforcer is the second-to-last player, he can force Avoider to pick $y$ eventually (this can also be seen as an application of Proposition \ref{prop:lastmove}, indeed Enforcer trivially wins on $H^{+y}$ which contains the empty edge). Let us now address the second assertion. The ``if" direction is simply Proposition \ref{prop:lastmove}, so assume Avoider has a winning strategy $\Sigma$ on $H$ as second-to-last player. We want to show that Avoider wins on $H^{-y}$ as last player. Note that the first player is the same in both games, since the last player and the parity are different. Therefore, and since $\Sigma$ obviously never has Avoider picking $y$, Avoider can simply apply $\Sigma$ playing on $H^{-y}$ as last player. By optimality of $\Sigma$, Avoider never fills an edge of $H$, so she never fills an edge of the subhypergraph $H^{-y}$ either.
\end{proof}

\subsection{Outcome of a disjoint union}

\noindent In Maker-Breaker games, it is easy to deduce the outcome of a disjoint union from the outcome of its components. For instance, if Maker wins as first or second player on $H_1$, then the same holds on any disjoint union $H_1 \cup H_2$, since Maker may disregard $H_2$ and play all her moves inside $H_1$. In this example, it is key that playing twice in a row in some component cannot harm Maker. Such arguments do not hold in Avoider-Enforcer games, which makes the study of disjoint unions less straightforward. However, we now show that the Avoider-Enforcer convention also has the convenient property that the outcome of a disjoint union is uniquely determined by the outcome of its components. We start by proving a useful lemma.

\begin{lemma}
\label{lem:many_moves}
    Let $H$ be a hypergraph.
    \begin{itemize}[nolistsep,noitemsep]
    	\item If a player (Avoider or Enforcer) wins on $H$ as first player, than that same player would have a winning strategy if their opponent was allowed to start by making any even number of moves.
    	\item If a player (Avoider or Enforcer) wins on $H$ as second player, than that same player would have a winning strategy if their opponent was allowed to start by making any odd number of moves.
    \end{itemize} 
\end{lemma}

\begin{proof}
    By Observation \ref{obs_dual}, it suffices to address Enforcer's side of things.
    \begin{itemize}[nolistsep,noitemsep]
    	\item Suppose that Enforcer has a winning strategy $\Sigma$ on $H$ as first player. Now consider the game played on $H$ where Avoider starts by playing on a subset $X$ of $V(H)$ with $|X|$ even. Enforcer first arbitrarily partitions $X$ into pairs $\{u_i,v_i\}$. Then, he ``forgets" that these vertices have been played, and he applies the strategy $\Sigma$. If, at some point during the game, $\Sigma$ requires Enforcer to play inside $X$ (say, some $u_i$), then he imagines that he picks $u_i$ and that Avoider answers by picking $v_i$. Since the set of vertices actually picked by Avoider is a superset of the imaginary one, which contains an edge since $\Sigma$ is a winning strategy, it also contains an edge.
    	\item Suppose that Enforcer wins on $H$ as second player: this means that, whatever Avoider's first move $x$, Enforcer wins on $H^{+x}$ as first player. Now, if Avoider makes an even number of moves additionally to $x$ (i.e. an odd total number of moves at the start including $x$), Enforcer still wins according to the even case which we have already addressed. \qedhere
    \end{itemize}
\end{proof}


\begin{theorem}\label{theo:union}
    Let $H_1$ and $H_2$ be two hypergraphs with disjoint vertex sets. The outcome of $H_1 \cup H_2$ is obtained from the outcomes of $H_1$ and $H_2$ according to the following table:

\begin{center}
\begin{tabular}{c|c|c|c}
  \diagbox{$H_2$}{$H_1$} & $\A$ & $\SL$ & $\E$  \\ \hline
  $\A$ & $\A$ & $\SL$ & $\E$ \\ \hline
  $\SL$ & $\SL$ & $\E$ & $\E$ \\ \hline
  $\E$ & $\E$ & $\E$ & $\E$ \\
\end{tabular}
\end{center}

\end{theorem}

\begin{proof}

The proof will distinguish cases depending on the outcomes of $H_1$ and $H_2$ as well as their parity. 

    \begin{enumerate}[label={\arabic*.}]
        \item $o(H_1) = \E$ (symmetric to the case $o(H_2) = \E$).

        Suppose that Enforcer is the last player on $H_1 \cup H_2$, and consider the following strategy. Enforcer plays arbitrary moves, all of them inside $H_2$. We can assume that, by the point all vertices of $H_2$ have been picked, Avoider has not yet filled an edge. Note that, since $o(H_1) = \E$, Enforcer has a winning strategy on $H_1$ as both first or second player. Therefore, according to Proposition~\ref{lem:many_moves}, he has a winning strategy on what remains of $H_1$, no matter the parity of the number of moves that Avoider has played in it. Enforcer can now follow this strategy to win the game.
        
        \item $o(H_1) = o(H_2) = \SL$.

        Suppose that Enforcer is the last player. We will give different strategies for Enforcer depending on the possible parities of $H_1$ and $H_2$:

        \begin{itemize}
            \item If $H_1$ and $H_2$ are both odd, then $H_1 \cup H_2$ is even, so Enforcer is the second player. Moreover, Enforcer has winning strategies $\Sigma_1$ and $\Sigma_2$ on $H_1$ and $H_2$ respectively as second player. Therefore, a winning strategy for Enforcer in $H_1 \cup H_2$ as second player is to play each move in the same hypergraph Avoider has just played in, according to $\Sigma_1$ or $\Sigma_2$ respectively. Avoider will fill an edge in one of the two hypergraphs: at the latest, this happens as she picks the last remaining vertex of $H_1$ or $H_2$.
            \item If $H_1$ is even and $H_2$ is odd, then $H_1 \cup H_2$ is odd, so Enforcer is the first player. Moreover, Enforcer has winning strategies $\Sigma_1$ on $H_1$ as first player and $\Sigma_2$ on $H_2$ as second player. Therefore, if Enforcer plays his first move in $H_1$ according to $\Sigma_1$ and then plays each subsequent move in the same hypergraph Avoider has just played in according to $\Sigma_1$ or $\Sigma_2$ respectively, Enforcer will win similarly to the previous case.
            	\\ The case where $H_1$ is odd and $H_2$ is even is symmetric.
            \item If $H_1$ and $H_2$ are both even, then $H_1 \cup H_2$ is even, so Enforcer is the second player. Moreover, Enforcer has winning strategies $\Sigma_1$ and $\Sigma_2$ on $H_1$ and $H_2$ respectively as first player. Consider the following strategy for Enforcer. Without loss of generality, assume that Avoider starts in $H_1$. Enforcer then plays his first move in $H_2$, according to $\Sigma_2$. We know from Proposition~\ref{lem:many_moves} that, if Avoider was the next player to pick a vertex in $H_1$, then Enforcer would have a winning strategy on the remainder of $H_1$. Therefore, Enforcer now plays each subsequent move in the same hypergraph Avoider has just played in, according to that strategy or to $\Sigma_2$ respectively. Avoider will fill an edge in one of the two hypergraphs: at the latest, this happens as she picks the last remaining vertex of $H_1$ or $H_2$.
        \end{itemize}

        In all cases, Enforcer wins on $H_1 \cup H_2$ as last player, $o(H_1 \cup H_2) = \E$.
        
        \item $o(H_1) = o(H_2) = \A$.

        Suppose that Avoider is the last player. We will give different strategies for Avoider depending on the possible parities of $H_1$ and $H_2$:
        
        \begin{itemize}
            \item If $H_1$ and $H_2$ are both even, then $H_1 \cup H_2$ is even, so Avoider is the second player. Moreover, Avoider has winning strategies $\Sigma_1$ and $\Sigma_2$ on $H_1$ and $H_2$ respectively as second player. Therefore, when playing on $H_1 \cup H_2$, Avoider can simply play each move in the same hypergraph Enforcer has just played in, according to $\Sigma_1$ or $\Sigma_2$ respectively. Since both hypergraphs are even, Avoider will play the last move in both of them, which ensures that no one plays twice in a row in $H_1$ or $H_2$. Therefore, for each $i \in \{1,2\}$, the game restricted to $H_i$ is a normal game where Avoider manages to avoid filling an edge.
            \item If $H_1$ is even and $H_2$ is odd, then $H_1 \cup H_2$ is odd, so Avoider is the first player. Moreover, Avoider has winning strategies $\Sigma_1$ on $H_1$ as second player and $\Sigma_2$ on $H_2$ as first player. Therefore, when playing on $H_1 \cup H_2$, Avoider can simply play her first move in $H_2$ according to $\Sigma_2$, and then play each subsequent move in the same hypergraph Enforcer has just played in, according to $\Sigma_1$ or $\Sigma_2$ respectively. Similarly to the previous case, she will avoid filling an edge in both hypergraphs and thus win the game.
            \\ The case where $H_1$ is odd and $H_2$ is even is symmetric.
            \item If $H_1$ and $H_2$ are both odd, then $H_1 \cup H_2$ is even, so Avoider is the second player. Moreover, Avoider has winning strategies $\Sigma_1$ and $\Sigma_2$ on $H_1$ and $H_2$ respectively as first player. Without loss of generality, assume Enforcer starts in $H_1$. Avoider plays her first move in $H_2$ according to $\Sigma_2$, and then plays each subsequent move in the same hypergraph Enforcer has just played in, using $\Sigma_2$ in $H_2$ and Proposition \ref{lem:many_moves} in $H_1$. Indeed, since the number of remaining vertices after the first round is even in both hypergraphs, Alice will play the last move in both, and the next move made in $H_1$ will be by Enforcer. After two Enforcer moves have been made in $H_1$, Avoider still has a winning strategy on what remains of $H_1$ by Proposition \ref{lem:many_moves}. In conclusion, except for the first two moves made in $H_1$ which that are taken care of by Proposition~\ref{lem:many_moves}, the order of moves in both hypergraphs will be normal and Avoider will avoid filling an edge in either.
            
        \end{itemize}

        In all cases, Avoider wins on $H_1 \cup H_2$ as last player, $o(H_1 \cup H_2) = \A$.
        
        \item $o(H_1) = \SL$ and $o(H_2) = \A$.

        First of all, we exhibit a winning strategy for Enforcer on $H_1 \cup H_2$ as second-to-last player. Enforcer plays arbitrary moves, all of them inside $H_2$. We can assume that, by the point all vertices of $H_2$ have been picked, Avoider has not yet filled an edge. It is now Enforcer's turn: he has not played a single move in $H_1$ yet, and the number of remaining vertices in $H_1$ is even since Enforcer is the second-to-last player. If $H_1$ is even, then Enforcer has a winning strategy on $H_1$ as first player and Avoider has played an even number of moves in $H_1$, so Enforcer wins according to Proposition~\ref{lem:many_moves}. If $H_1$ is odd, then Enforcer has a winning strategy on $H_1$ as second player and Avoider has played an odd number of moves in $H_1$, so Enforcer also wins according to Proposition~\ref{lem:many_moves}.

        Finally, it remains to show that Avoider wins on $H_1 \cup H_2$ as second-to-last player. To do so, we will distinguish between the possible parities of $H_1$ and $H_2$: 

        \begin{itemize}

        	\item If $H_1$ is odd and $H_2$ is even, then $H_1 \cup H_2$ is odd, so Avoider is the second player. Moreover, Avoider has winning strategies $\Sigma_1$ and $\Sigma_2$ on $H_1$ and $H_2$ respectively as second player. When playing on $H_1 \cup H_2$, Avoider can, as long as possible, play each move in the same hypergraph Enforcer has just played in, according to $\Sigma_1$ or $\Sigma_2$ respectively. First suppose that all vertices in $H_2$ are picked before all vertices in $H_1$ are: in this case, parity ensures that Avoider can perform this strategy until the end, thus avoiding filling an edge in either hypergraph and winning the game. Now, suppose that Enforcer picks the last remaining vertex in $H_1$ before all vertices in $H_2$ have been picked. In particular, Avoider has successfully avoided filling an edge so far, and it is now her turn. The remainder of $H_2$ at this point is even, and Avoider has a winning strategy on it as second player, so she also has a winning strategy on it as first player by Theorem \ref{theo:last}. Therefore, Avoider can follow that strategy to win the game.

            \item If $H_1$ and $H_2$ are both even, then $H_1 \cup H_2$ is even, so Avoider is the first player. Since $o(H_1)=\SL$ and $H_1$ is even, there exists $x \in V(H_1)$ such that Avoider has a winning strategy as second player on $H_1^{+x}$.
            Note that Avoider does not have a winning strategy as first player on $H_1^{+x}$ by Proposition \ref{lem:many_moves}, so $o(H_1^{+x})=\SL$. Since $H_1^{+x}$ is odd and $H_2$ is even, the odd/even case which we have already addressed yields $o(H_1^{+x} \cup H_2)=\SL$, so Avoider has a winning strategy on $H_1 \cup H_2$ as second-to-last player where she picks $x$ as her first move.

            \item If $H_1$ and $H_2$ are both odd, then $H_1 \cup H_2$ is even, so Avoider is the first player. Since $o(H_2)=\A$ and $H_2$ is odd, there exists $x \in V(H_2)$ such that $o(H_2^{+x})=\A$. Since $H_1$ is odd and $H_2^{+x}$ is even, the odd/even case which we have already addressed yields $o(H_1 \cup H_2^{+x})=\SL$, so Avoider has a winning strategy on $H_1 \cup H_2$ as second-to-last player where she picks $x$ as her first move.
        
            \item If $H_1$ is even and $H_2$ is odd, then $H_1 \cup H_2$ is odd, so Avoider is the second player. As his first move, Enforcer picks $y \in V(H_i)$ for some $i \in \{1,2\}$. Note that, since Avoider has a winning strategy on $H_i$ as first player, she has a winning strategy on $H_i^{-y}$ as second player by Proposition \ref{lem:many_moves}, so $o(H_1^{-y})=\SL$ if $i=1$ or $o(H_2^{-y})=\A$ if $i=2$. Now, since Avoider has a winning strategy on $H_{3-i}$ as first player, there exists $x \in V(H_{3-i})$ such that Avoider has a winning strategy on $H_{3-i}^{+x}$ as second player. This means $o(H_1^{+x})=\SL$ if $i=2$ or $o(H_2^{+x})=\A$ if $i=1$. All in all, the odd/even case which we have already addressed yields $o(H_i^{-y} \cup H_{3-i}^{+x})=\SL$, so Avoider has a winning strategy after picking $x$. \qedhere
            
        \end{itemize}

    \end{enumerate}
\end{proof}

\begin{remark}
	When $o(H_1) = o(H_2) = \E$, not only does Enforcer win on $H_1 \cup H_2$ as last player, but he can actually force Avoider to fill two edges, one in each hypergraph. Indeed, we have $o(H_1^T) = o(H_2^T) = \A$ by Observation \ref{obs_dual}, hence $o(H_1^T \cup H_2^T)=\A$ by Theorem \ref{theo:union}, so Enforcer as last player can avoid filling a transversal in $H_1$ and $H_2$ at once, meaning Avoider fills an edge in both.
\end{remark}

\begin{corollary}\label{coro:isolated}
	Adding or removing any number of isolated vertices in a hypergraph does not modify the outcome.
\end{corollary}

\begin{proof}
	An isolated vertex obviously has outcome $\A$, which we know is neutral for the disjoint union by Theorem \ref{theo:union}.
\end{proof}

\subsection{Some basic strategic principles}

\noindent A crucial property of the Maker-Breaker convention is subhypergraph monotonicity: if Maker wins on some subhypergraph of $H$, then Maker also wins on $H$. A similar result holds for Enforcer in the Avoider-Enforcer convention. Let us specify that we will never consider induced subhypergraphs in this paper: a subhypergraph of $H$ is anything that can be obtained from $H$ by removing vertices and/or edges.

\begin{proposition}[Monotonicity Principle]\label{prop:monotonicity}
	Let $H$ be a hypergraph, and let $H'$ be a subhypergraph of $H$. If Enforcer wins on $H'$ as last player (resp. as second-to-last player), then Enforcer wins on $H$ as last player (resp. as second-to-last player).
\end{proposition}

\begin{proof}
	It is possible to build $H$ from the subhypergraph $H'$ in two steps, by adding the missing vertices first and then the missing edges. The first step is outcome-neutral by Corollary \ref{coro:isolated}. The second step cannot harm Enforcer: obviously, a winning Enforcer strategy on a given hypergraph also wins on any hypergraph with the same vertex set and additional edges.
\end{proof}

\noindent From Avoider's side of things, pairing strategies are a simple but powerful tool \cite{Hefetz2007}. A {\em pairing} in a hypergraph $H$ is a set $\Pi$ of pairwise disjoint pairs in $V(H)$ such that, for all $e \in E(H)$, there exists $\pi \in \Pi$ satisfying $\pi \subseteq e$. 

\begin{proposition}[Pairing Principle]\label{prop:pairing}
	Let $H$ be a hypergraph. If $H$ admits a pairing, then $o(H)=\A$.
\end{proposition}

\begin{proof}
	Let $\Pi$ be a pairing of $H$. It suffices to address the case where Avoider is the last player. Her strategy is as follows: whenever Enforcer has just picked some $y \in \{x,y\} \in \Pi$, she answers by picking $x$, otherwise she picks an arbitrary vertex outside $\Pi$. Since $\Pi$ involves an even number of vertices and Avoider is the last player, she can always follow this strategy and Enforcer is forced to play first inside each pair from $\Pi$. At the end of the game, Enforcer has picked at least one (actually: exactly one) vertex in each pair from $\Pi$, which ensures that he has filled a transversal of $H$ by definition of a pairing. Therefore, Avoider wins.
\end{proof}

\noindent This principle may also be applied locally, for instance if we identify a pair of vertices that are equivalent in the following sense:

\begin{definition}
	Let $H$ be a hypergraph, and let $x,y \in V(H)$ be distinct. We say $x$ and $y$ are {\em indistinguishable} in $H$ if the following two conditions are both satisfied:
	\begin{itemize}[noitemsep,nolistsep]
		\item For all $e \in E(H)$ such that $x \in e$ and $y \not\in e$, we have $(e \setminus \{x\}) \cup \{y\} \in E(H)$.
		\item For all $e \in E(H)$ such that $y \in e$ and $x \not\in e$, we have $(e \setminus \{y\}) \cup \{x\} \in E(H)$.
	\end{itemize}
\end{definition}

If two vertices are indistinguishable, we can assume that the two players will each pick one of the vertices. This is the principle of the so-called “Super Lemma”, proved in Oijid's PhD thesis \cite{oijid2024}. For the sake of completeness, we give a proof here.
Note that if $x$ and $y$ are indistinguishable vertices of an hypergraph $H$, then $H^{+x-y}=H^{+y-x}$. This hypergraph corresponds to the hypergraph  obtained from $H$ as follows: $x$ and $y$ are removed, any edge containing both $x$ and $y$ is removed, and any edge $e$ containing exactly one of $x$ or $y$ is replaced by $e \setminus \{x,y\}$.

\begin{proposition}[Super Lemma \cite{oijid2024}]\label{prop:reduced}
	Let $H$ be a hypergraph, and let $x,y \in V(H)$ be indistinguishable. Then $o(H)=o(H^{+x-y})$.
\end{proposition}

\begin{proof} Let $H'=H^{+x-y}=H^{+y-x}$.
	Fix an order (i.e. choose the last player), and let P be the player who has a winning strategy $\Sigma$ on $H'$ for that order. We want to show that P also has a winning strategy on $H$ for the same order. Note that the first player is the same in both games since $H$ and $H'$ have the same parity. Therefore, P can play a fictitious game on $H'$ following the strategy $\Sigma$ and copy the moves in the real game played on $H$, except when the opponent picks $x$ or $y$ in which case P picks the other. If $x$ and $y$ are the only vertices remaining in the real game and P is second-to-last, then P picks one arbitrarily. In all cases, the vertices $x$ and $y$ end up being shared between both players in the real game: by definition of $H'$, this implies that Avoider fills an edge of $H$ in the real game if and only if Avoider fills an edge of $H'$ in the fictitious game. By optimality of $\Sigma$, this ensures that P wins.
\end{proof}

%

\noindent Thanks to Proposition \ref{prop:reduced}, hypergraphs can always be assumed to have no pairs of indistinguishable vertices, up to some polynomial-time preprocessing.

\section{Hypergraphs of rank~2}\label{sec:rank2}

\noindent In this section, we solve the game on hypergraphs of rank 2. Note that, by Proposition \ref{prop:1-edge}, the case of hypergraphs of rank 2 reduces to the case of 2-uniform hypergraphs i.e. graphs. Therefore, we will only consider graphs in this section.

\subsection{Some elementary graphs and their outcome}

\noindent We recall some vocabulary from the literature which will be useful, along with some notations. For any graph $G$, we denote by $2G$ the graph formed by two vertex-disjoint copies of $G$. We denote by $P_n$ and $C_n$ the {\em path} and the {\em cycle} on $n$ vertices respectively. If $x_1,\ldots,x_n$ are consecutive vertices along a path $P_n$ (resp. a cycle $C_n$), then that path (resp. cycle) may be denoted as $x_1 x_2 \cdots x_n$. The {\em bull} is the graph obtained by adding a pendent edge to two vertices of $C_3$ (see Figure \ref{fig:example_graphs}, left). The {\em 3-sunlet} is the graph obtained by adding a pendent edge to all vertices of $C_3$ (see Figure \ref{fig:example_graphs}, middle). A {\em pseudo-star} is any connected graph $G$ such that $|V(G)| \geq 3$ and containing a vertex $y$ such that all connected components of $G^{-y}$ are of size 1 or 2. An example pseudo-star is given on the right of Figure \ref{fig:example_graphs}. Note that $C_3$, $P_3$, $P_4$ and $P_5$ are pseudo-stars.

\begin{figure}[h]
	\centering
	\includegraphics[scale=.55]{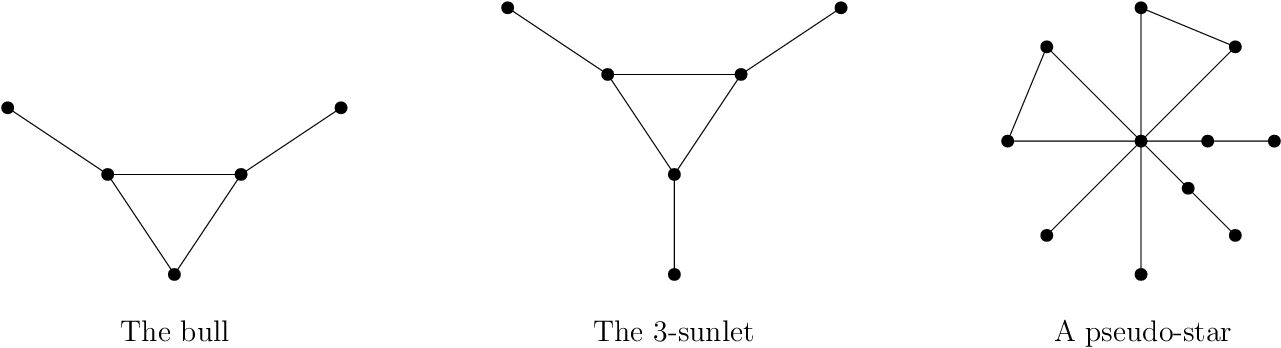}
	\caption{Some elementary graphs.}\label{fig:example_graphs}
\end{figure}

\begin{proposition}\label{prop:graphs}
    $P_1$ and $P_2$ have outcome $\A$. The bull, $C_5$ and pseudo-stars have outcome $\SL$. The 3-sunlet, $C_4$ and $2P_3$ have outcome $\E$.
\end{proposition}

\begin{proof}
	Clearly, $o(P_1)=o(P_2)=\A$. It is also straightforward that $o(P_3)=\SL$: indeed, Avoider wins as second-to-last player because she only picks one vertex in total, and Enforcer wins as second-to-last player by picking a vertex of degree 1 whatever Avoider's first pick is. As a consequence, $o(2P_3)=\E$ by Theorem \ref{theo:union}. Since $C_4$, $C_5$, the bull, the 3-sunlet and pseudo-stars all contain $P_3$ as a subgraph, Enforcer wins on all these graphs as second-to-last player by Corollary \ref{prop:monotonicity} (Monotonicity Principle), so assume from now that Avoider is the second-to-last player.
	\begin{itemize}
	    \item[--] Let $xyzt$ be a $C_4$. Since $C_4$ is even, Avoider plays first. By symmetry, we can assume that Avoider starts by picking $x$. Enforcer then picks $z$, and Avoider loses since her next pick will fill the edge $xy$ or the edge $xt$.
	    \item[--] Let $xyztu$ be a $C_5$. Since $C_5$ is odd, Enforcer plays first. By symmetry, we can assume that Enforcer starts by picking $x$. Avoider then picks $y$, and wins by picking either $t$ or $u$ in the second round.
	    \item[--] Consider a bull $B$ formed by a $C_3$ $xyz$ plus edges $xu$ and $yv$. Since $B$ is odd, Enforcer plays first. After Enforcer's first pick, $x$ and $u$ are both free, or $y$ and $v$ are both free. By symmetry, assume the former. Avoider then picks $u$, and wins by not picking $x$ in the second round.
	    \item[--] Consider a 3-sunlet $S$ formed by a $C_3$ $xyz$ plus edges $xu$, $yv$ and $zw$. Since $S$ is even, Avoider plays first. If Avoider plays inside the cycle before Enforcer does, then she loses, since Enforcer will then force her to pick one of the other two vertices of the cycle eventually. Therefore, by symmetry, assume Avoider starts by picking $u$. Enforcer then picks $v$, and Avoider picks $w$. Enforcer finally picks $y$, and Avoider loses since her next pick will fill the edge $xu$ or the edge $zw$.
	    \item[--] Lastly, consider a pseudo-star $G$. By definition, there exists a vertex $y$ such that each component of $G^{-y}$ is a $P_1$ or a $P_2$. Therefore, the edges of $G^{-y}$ form a pairing of $G^{-y}$, so $o(G^{-y})=\A$ by Proposition \ref{prop:pairing} (Pairing Principle). Proposition \ref{prop:lastmove} thus ensures that Avoider wins on $G$ as second-to-last player. \qedhere
	\end{itemize}
\end{proof}

\subsection{Complete structural characterisation of the outcome}

\noindent Using the previous elementary examples, we are going to establish the following result, which solves the game for general graphs.

\begin{theorem}\label{theo:graphs}
    Let $G$ be a graph.
    \begin{itemize}[noitemsep,nolistsep]
		\item If $G$ does not contain $P_3$ as a subgraph, then $o(G)=\A$;
		\item If $G$ contains $2P_3$, $C_4$ or the 3-sunlet as a subgraph, then $o(G)=\E$;
		\item In all other cases, $o(G)=\SL$.
	\end{itemize}
\end{theorem}

\begin{corollary}\label{coro:graphs}
    The outcome of the Avoider-Enforcer game can be computed in polynomial time on hypergraphs of rank 2.
\end{corollary}

\noindent To prove Theorem \ref{theo:graphs}, we discuss depending on who is the last player. For both cases, we give two equivalent characterisations of the winner: one that describes the connected components (extensional) and one by forbidden subgraphs (intensional). Let us start with the case where Avoider is the last player, which is straightforward.

\begin{proposition}\label{prop:graphs_A}
	Let $G$ be a graph. The following three assertions are equivalent:
	\begin{enumerate}[noitemsep,nolistsep,label={\textup{(\arabic*)}}]
		\item Avoider wins on $G$ as last player.
		\item Every connected component of $G$ is $P_1$ or $P_2$.
		\item $G$ does not contain $P_3$ as a subgraph.
	\end{enumerate}
\end{proposition}

\begin{proof}
	It is clear that (3)$\implies$(2). Since Avoider loses on $P_3$ as last player by Proposition \ref{prop:graphs}, we also have (1)$\implies$(3) by Proposition \ref{prop:monotonicity} (Monotonicity Principle). Finally, if each connected component of $G$ is $P_1$ or $P_2$, then the edges of $G$ form a pairing of $G$ hence $o(G)=\A$ by Proposition \ref{prop:pairing} (Pairing Principle), so (2)$\implies$(1).
\end{proof}

\noindent The case where Enforcer is the last player is not difficult either, though slightly more involved. The key is the following lemma, which is purely a graph theory result.

\begin{lemma}\label{lem:graphs}
	The connected graphs that do not contain $2P_3$, $C_4$ or the 3-sunlet as a subgraph are exactly the following:
	\begin{enumerate}[noitemsep,nolistsep,label={\textup{\arabic*.}}]
		\item $P_1$;
		\item $P_2$;
		\item $C_5$;
		\item the bull;
		\item pseudo-stars.
	\end{enumerate}
\end{lemma}

\begin{proof}
	Let $G$ be a connected graph that does not contain $2P_3$, $C_4$ or the 3-sunlet as a subgraph. We distinguish between three cases:
	\begin{itemize}
		\item Firstly, suppose $G$ is a path $P_n$. Since $G$ contains no $2P_3$, we have $n \leq 5$. If $n=1$ (resp. $n=2$), then we are in Case 1 (resp. Case 2). If $n \in \{3,4,5\}$, then we are in Case 5.
		\item Secondly, suppose $G$ is a cycle $C_n$. Since $G$ contains no $2P_3$, we have $n \leq 5$. Moreover, since $G$ contains no $C_4$, we have $n \neq 4$. If $n=3$, then we are in Case 5. If $n=5$, then we are in Case 3.
		\item Lastly, suppose $G$ is neither a path nor a cycle. Since $G$ is connected, this means $G$ contains a vertex $x$ of degree at least 3. Let $y,z,t$ be three distinct neighbours of $x$. If $G^{-x}$ contains no $P_3$, then we are in Case 5. Therefore, we now assume that $G^{-x}$ contains a path $P$ on three vertices. Define $m = |\{y,z,t\} \cap V(P)|$.
			\begin{itemize}[noitemsep,nolistsep]
				\item[--] It is impossible that $m=3$. Indeed, if we had $P=yzt$ for instance, then $xyzt$ would be a $C_4$ in $G$.
				\item[--] It is impossible that $m \leq 1$. Indeed, if we had $y,z \not\in V(P)$ for instance, then $yxz \cup P$ would be a $2P_3$ in $G$.
				\item[--] In conclusion, we have $m=2$. Without loss of generality, assume that $V(P)=\{u,y,z\}$ for some $u \not\in\{x,y,z,t\}$. It is impossible that $P=yuz$, because $xyuz$ would be a $C_4$ in $G$. Therefore, we have $P=uyz$ or $P=yzu$: by symmetry, assume $P=uyz$. We get a bull $B$ in $G$ formed by the $C_3$ $xyz$ and the pendant edges $xt$ and $yu$. Finally, it is impossible that $G$ contains a sixth edge intersecting $V(B)$, because a $2P_3$, $C_4$ or 3-sunlet would appear in all cases: see Figure \ref{fig:cases_bull}. Since $G$ is connected, this means $G=B$, so we are in Case 4.

					\begin{figure}[h]
						\centering
						\includegraphics[scale=.5]{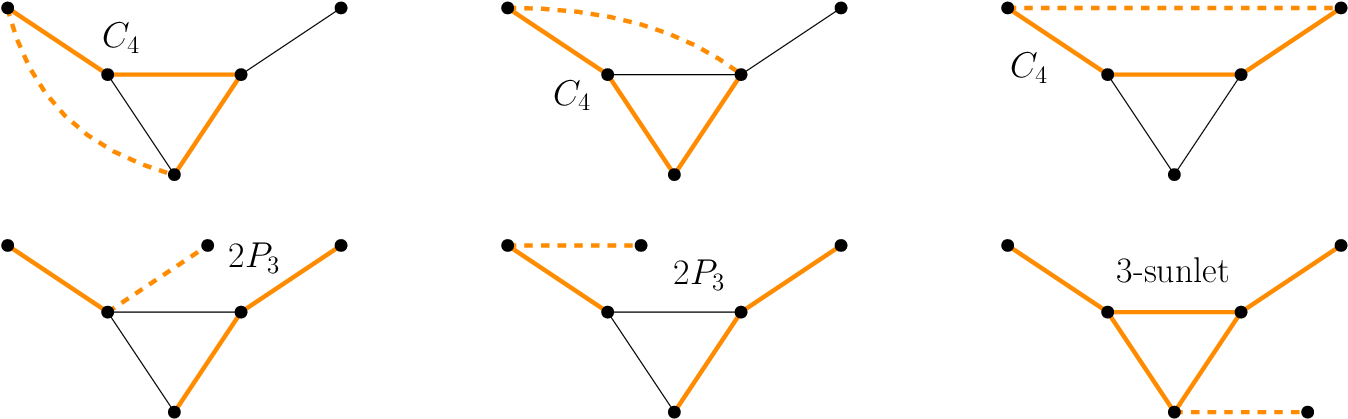}
						\caption{All possibilities of adding an edge (dashed line) connected to a bull, up to symmetries. The forbidden subgraph thus created is highlighted.}\label{fig:cases_bull}
					\end{figure}
					
			\end{itemize}
	\end{itemize}
	Conversely, it is clear that $P_1$, $P_2$, $C_5$, the bull and pseudo-stars are connected graphs which do not contain $2P_3$, $C_4$ or the 3-sunlet as a subgraph. This ends the proof.
\end{proof}

\begin{proposition}\label{prop:graphs_E}
	Let $G$ be a graph. The following three assertions are equivalent:
	\begin{enumerate}[noitemsep,nolistsep,label={\textup{(\arabic*)}}]
		\item Avoider wins on $G$ as second-to-last player.
		\item Every connected component of $G$ is $P_1$ or $P_2$, except for at most one which is either: $C_5$, the bull or a pseudo-star.
		\item $G$ does not contain $2P_3$, $C_4$ or the 3-sunlet as a subgraph.
	\end{enumerate}
\end{proposition}
 
\begin{proof}
	Since Avoider loses on $2P_3$, $C_4$ and the 3-sunlet as second-to-last player by Proposition \ref{prop:graphs}, we have (1)$\implies$(3) by Proposition \ref{prop:monotonicity} (Monotonicity Principle). If $G$ does not contain $2P_3$, $C_4$ or the 3-sunlet as a subgraph, then each connected component of $G$ is a $P_1$, $P_2$, $C_5$, bull or pseudo-star by Lemma \ref{lem:graphs}, moreover at most one is a $C_5$, bull or pseudo-star otherwise we would get a $2P_3$, so (3)$\implies$(2). Finally, recall that $o(P_1)=o(P_2)=\A$ and $o(C_5)=o(\text{bull})=o(\text{pseudo-star})=\SL$ by Proposition \ref{prop:graphs}. Therefore, using Theorem \ref{theo:union}, a disjoint union of copies of $P_1$ and $P_2$ has outcome $\A$, and adding one component which is a $C_5$, bull or pseudo-star makes the outcome $\SL$. Since $\A$ and $\SL$ are the two outcomes where Avoider wins as second-to-last player, we get (2)$\implies$(1).
\end{proof}
 
\begin{proof}[Proof of Theorem \ref{theo:graphs} and Corollary \ref{coro:graphs}]
    Theorem \ref{theo:graphs} clearly follows from Propositions \ref{prop:graphs_A} and \ref{prop:graphs_E}. Since testing whether a graph contains a fixed given subgraph can be done in polynomial time, Corollary \ref{coro:graphs} ensues.
\end{proof}

\section{Linear hypergraphs of rank 3}\label{section_rank3}

\noindent After settling hypergraphs of rank 2, a natural first step is to address linear hypergraphs of rank 3, similarly to what was done for the Maker-Breaker convention \cite{kutz}. In this section, we solve the game on linear hypergraphs of rank 3 when Avoider is the last player.

\subsection{Leaf-edges and reduced hypergraphs}

\noindent The most basic linear hypergraphs are linear trees. Let us start by observing some of their properties.

\begin{proposition}\label{prop:leafedge}
    Let $T$ be a linear tree.
    \begin{itemize}[noitemsep,nolistsep]
    \item If $T$ has no leaf-edge, then $T$ is an isolated vertex.
    \item If $T$ has exactly one leaf-edge, then $T$ has a single edge.
    \item If $T$ has exactly two leaf-edges $e$ and $e'$, then $T$ forms a linear simple walk $(e,\ldots,e')$.
    \end{itemize}
\end{proposition}

\begin{proof}
    The first two assertions are straightforward. Let us prove the third one by induction on the number of edges. Clearly, if a linear tree has exactly two edges, then both are leaf-edges and they form a linear simple walk. Now, let $m \geq 2$ such that the property holds, and let $T$ be a linear tree with exactly two leaf-edges, obtained from an isolated vertex by adding edges $e_1,\ldots,e_{m+1}$ successively. Denote by $T_i$ the linear tree with edges $e_1,\ldots,e_i$. Since $T_2$ and $T_{m+1}=T$ have exactly two leaf-edges, and the number of leaf-edges never decreases during the process of adding edges, $T_m$ must also have exactly two leaf-edges $e_r$ and $e_s$ for some $1 \leq r,s \leq m$. Now, by the induction hypothesis, $T_m$ forms a linear simple walk $(e_r\ldots,e_s)$. As $e_{m+1}$ is a leaf-edge of $T_{m+1}$, one of $e_r$ and $e_s$ is not a leaf-edge of $T_{m+1}$. By symmetry, assume $e_s$ is not a leaf-edge of $T_{m+1}$: we have $e_{m+1} \cap e_s =\{x\}$ where $x$ is of degree 1 in $T_m$. Since $|e_{m+1} \cap V(T_m)|=1$, this implies that $e_s$ is the only edge of $T_m$ that intersects $e_{m+1}$, so $(e_r\ldots,e_s,e_{m+1})$ is a linear simple walk.
\end{proof}

\noindent Note that two vertices of degree 1 in the same leaf-edge of a hypergraph are indistinguishable, so Proposition \ref{prop:reduced} allows us to remove these two vertices and that leaf-edge without altering the outcome. Therefore, we introduce the following property, which we will freely assume throughout our study of the game on linear hypergraphs of rank 3.

\begin{definition}
	A hypergraph of rank 3 is {\em reduced} if it contains no leaf-edge of size 3.  
	The {\em reduced hypergraph} of a hypergraph $H$ is obtained from $H$ as follows: remove every leaf-edge $e \in E(H)$ of size 3 as well as two vertices of degree 1 of $e$, then repeat the process on the remaining hypergraph until there is no leaf-edge anymore. 
\end{definition}

\noindent Note that the definition of a linear tree yields the following property.

\begin{observation}\label{obs:hypreduced}
Let $H$ be a 3-uniform hypergraph. The reduced hypergraph of $H$ is an isolated vertex if and only if $H$ is a linear tree.
\end{observation}


\subsection{Some elementary linear hypergraphs of rank 3}

\begin{definition}
	We define the following linear hypergraphs of rank 3:
	\begin{itemize}[noitemsep,nolistsep]
		\item For distinct vertices $a$ and $b$, an {\em $ab$-chain} is a hypergraph $P$ which forms a linear simple walk $(e_1,\ldots,e_{\ell})$ in which $|e_i|=3$ for all $1 \leq i \leq \ell$ and $e_1$ (resp. $e_{\ell}$) is the only edge containing $a$ (resp. $b$). Such a walk is said to {\em represent} $P$. We say $\ell=|E(P)|\geq 1$ is the {\em length} of $P$. A {\em chain} is an $ab$-chain for some vertices $a$ and $b$.
		\item A {\em nunchaku} \cite{MBrank3} is a hypergraph $N$ formed by the edges of a linear simple walk $(e_1,\ldots,e_{\ell})$ in which $|e_i|=3$ for all $2 \leq i \leq \ell-1$ and $|e_1|=|e_{\ell}|=2$. We say $\ell=|E(N)|\geq 2$ is the {\em length} of $N$.
		\item A {\em cycle} is a hypergraph $C$ obtained from an $ab$-chain $P$ of length at least 2 by adding a 3-edge $\{a,b,c\}$ where $c \not\in V(P)$. We say $\ell=|E(C)|\geq 3$ is the {\em length} of $C$.
		\item The {\em prism} is the hypergraph $H$ obtained from a cycle $C$ of length 3 by adding a 3-edge containing the vertices that are of degree 1 in $C$.
	\end{itemize}
	See Figure \ref{fig:examples_rank3} for some illustrations.
\end{definition}

\begin{figure}[h]
	\centering
	\includegraphics[scale=.5]{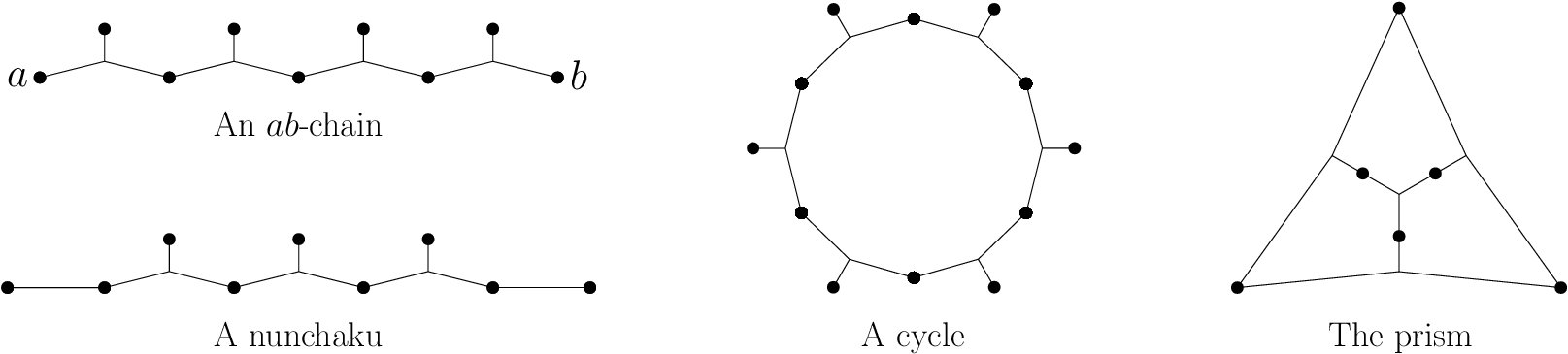}
	\caption{Some elementary linear hypergraphs of rank 3.}\label{fig:examples_rank3}
\end{figure}

\noindent Given a connected 3-uniform linear hypergraph $H$ and two distinct vertices $a,b \in V(H)$, we denote by $\dist_H(a,b)$ the length of a shortest $ab$-chain in $H$. The following lemma ensures in particular that $\dist_H$ is a metric on $V(H)$.

\begin{lemma}\label{lem:chains}
	Let $H$ be a linear hypergraph, and let $a,b,c \in V(H)$ be pairwise distinct. Let $P$ be an $ab$-chain and let $P'$ be a $bc$-chain. Then $P \cup P'$ contains an $ac$-chain.
\end{lemma}

\begin{proof}
	Let $(e_1,\ldots,e_{\ell})$ and $(e'_1,\ldots,e'_m)$ be linear simple walks representing $P$ and $P'$ respectively, with $a \in e_1$,  $b \in e_{\ell} \cap e'_1$ and $c \in e'_m$. If $c \in V(P)$, then we can define $1 \leq i \leq \ell$ as the minimum index such that $c \in e_i$, and the sequence $(e_1,\ldots,e_i)$ clearly is a linear simple walk which represents an $ac$-chain. A similar argument concludes if $a \in V(P')$. Therefore, we may assume that $a \not\in V(P')$ and $c \not\in V(P)$. Since $b \in V(P')$, we can define $1 \leq i \leq \ell$ as the minimum index such that $e_i \cap V(P') \neq \varnothing$, and $1 \leq j \leq m$ as the maximum index such that $e_i \cap e'_j \neq \varnothing$. The definitions of $i$ and $j$ ensure that the sequence $(e_1,\ldots,e_i,e'_j,\ldots,e'_m)$ is a simple walk, moreover it is linear since $H$ is. Finally, since $a \not\in V(P')$ and $c \not\in V(P)$, the only edge of this walk containing $a$ is $e_1$ and the only edge containing $c$ is $e'_m$, therefore it represents an $ac$-chain.
\end{proof}


\noindent We will see that nunchakus, prisms and cycles play a major role for the Avoider-Enforcer game on hypergraphs of rank 3. We now determine their outcome.

\begin{proposition}\label{prop:outcome_rank3}
Nunchakus have outcome $\SL$. The prism and cycles have outcome $\A$.
\end{proposition}

\begin{proof}
	We first prove the result on nunchakus when Avoider is the last player (note that Avoider is also the first player since nunchakus are odd). We proceed by induction on length. A nunchaku of length 2 is simply a $P_3$ graph, on which we already know Avoider loses as last player. Now, let $N$ be a nunchaku of length $\ell \geq 3$, and assume that Avoider loses as last player on all nunchakus of length at most $\ell-1$. Let $a$ and $b$ be the vertices of degree 1 in the 2-edges of $N$. If Avoider's first pick is in one of the 2-edges, then she loses since this creates a 1-edge. Otherwise, Avoider's first pick $x$ creates two new nunchakus $N_1$ and $N_2$, both shorter than $N$, one of which does not contain $a$ and the other does not contain $b$ (see Figure \ref{fig:nunchaku_induction}). Enforcer can respond by picking $y \in \{a,b\}$, so that $N_i$ is a subhypergraph of the updated hypergraph $N^{+x-y}$ for some $i \in \{1,2\}$. The induction hypothesis ensures that Enforcer wins on $N_i$ as second-to-last player, so Enforcer wins on $N^{+x-y}$ as second-to-last player by Proposition \ref{prop:monotonicity} (Monotonicity Principle), therefore Enforcer wins on $N$ as second-to-last player.
	
	\begin{figure}[h]
		\centering
		\includegraphics[width=\textwidth]{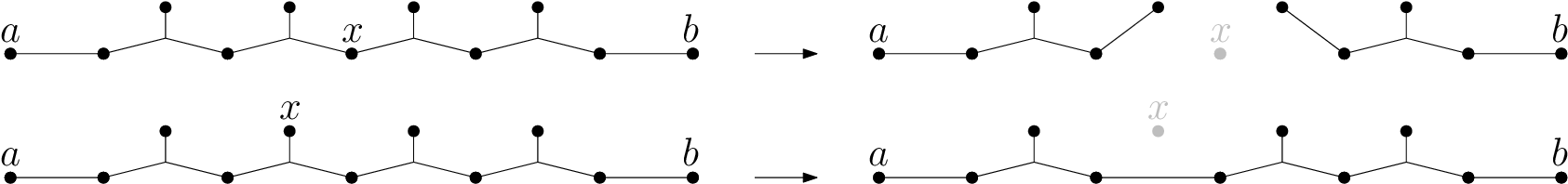}
		\caption{Top: if $x$ is of degree 2 in $N$ (left) then $N^{+x}$ contains two disjoint nunchakus (right). Bottom: if $x$ is of degree 1 in $N$ (left) then $N^{+x}$ contains two nunchakus sharing a 2-edge (right).}\label{fig:nunchaku_induction}
	\end{figure}

	\noindent To show that Avoider wins on any nunchaku $N$ as second-to-last player, it suffices to show that there exists $y\in V(N)$ such that Avoider wins on $N^{-y}$ as last player, according to Proposition \ref{prop:lastmove}. Let $y$ be a vertex of degree 1 in a 2-edge of $N$: since $N^{-y}$ admits a pairing, as illustrated in Figure \ref{fig:nunchaku_pairing}, we know $o(N^{-y})=\A$ by Proposition \ref{prop:pairing} (Pairing Principle).
	
	\begin{figure}[h]
		\centering
		\includegraphics[width=\textwidth]{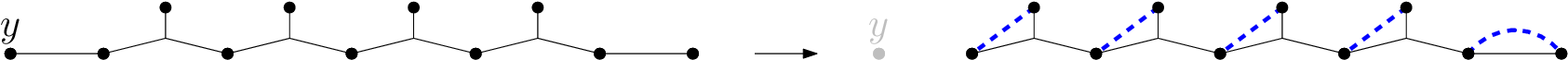}
		\caption{$N$ is on the left, $N^{-y}$ is on the right with a pairing being represented.}\label{fig:nunchaku_pairing}
	\end{figure}
	
	\noindent It is clear that a cycle admits a pairing, where each vertex of degree 1 is paired with a neighbour of degree 2, so Avoider likewise wins on any cycle as last player. Finally, let $H$ be a prism, and suppose Avoider is the last player: since $H$ is even, this means Enforcer is the first player. Enforcer picks some vertex $y$, and Avoider picks some $x \in N(y)$ (the set of neighbours of $y$). Since $x$ is of degree 2 in $H$ (as is any vertex in a prism), this choice ensures that the updated hypergraph $H^{-y+x}$ contains only one 2-edge $e$. There are only two rounds of play left, so Alice simply has to pick some $z \not\in e$ in the next round to win.
\end{proof}

\subsection{Structural characterisation of the outcome when Avoider is the last player}

\noindent We are going to establish the following result.

\begin{theorem}\label{theo:3unif_A}
	Let $H$ be a reduced linear hypergraph of rank 3. The following two assertions are equivalent:
	\begin{enumerate}[noitemsep,nolistsep,label={\textup{(\arabic*)}}]
		\item Avoider wins on $H$ as last player.
		\item Each connected component of $H$ is a cycle, a prism, an isolated 2-edge or an isolated vertex.
	\end{enumerate}
\end{theorem}

\begin{corollary}\label{coro:3unif_A}
    Deciding whether a linear hypergraph of rank 3 has outcome $\A$ can be done in polynomial time.
\end{corollary}

\noindent First of all, notice that addressing the connected case is sufficient: indeed, by Theorem \ref{theo:union}, Avoider wins as last player if and only if Avoider wins on each connected component as last player. Also note that the result is trivial if there exists a 0-edge or a 1-edge, since Avoider then loses as last player.
\\ Enforcer's main idea is to rely on Avoider creating a nunchaku: since a nunchaku has outcome $\SL$ according to Proposition \ref{prop:outcome_rank3}, this would ensure Enforcer's win as second-to-last player. This happens more easily if Avoider is the first player. Therefore, we distinguish two cases depending on parity.


\subsubsection{Odd case: Avoider is first and last}

\noindent We can first address the case where there exists a 2-edge.

\begin{proposition}\label{prop:3unif_A_odd}
	Avoider loses as last player on any odd connected linear hypergraph of rank 3 which contains a 2-edge.
\end{proposition}

\begin{proof}
	If Avoider plays her first move inside a 2-edge, then this creates a 1-edge. Otherwise, linearity and connectedness ensure that there exists a chain between the vertex that Avoider picks and a 2-edge, so Avoider's first move creates a nunchaku. In all cases, we get a subhypergraph that has outcome $\SL$, so Enforcer wins as second-to-last player by Proposition \ref{prop:monotonicity} (Monotonicity Principle).
\end{proof}

\noindent In the 3-uniform case, Avoider needs to play twice before a nunchaku can appear. After Avoider plays her first move, Enforcer wants to pick a non-cut vertex of the starting hypergraph, so that Avoider's second pick will necessarily be connected to the first by a chain. For this, the following lemma (which is not game-related in itself) is key.

\begin{lemma}\label{lem:NonCut1}
	Any connected reduced linear 3-uniform hypergraph which is not an isolated vertex has at least two non-cut vertices.
\end{lemma}

\noindent Lemma \ref{lem:NonCut1} is an immediate consequence of the following result, which details how to find a non-cut vertex.

\begin{lemma}\label{lem:NonCut2}
	Let $H$ be a connected reduced linear 3-uniform hypergraph which is not an isolated vertex, and let $x$ be an arbitrary vertex of $H$. Out of all vertices that maximise $\dist_H(x,\cdot\,)$, let $u$ be one with minimum degree. Then $u$ is a non-cut vertex of $H$.
\end{lemma}

 \begin{proof}

	Suppose for a contradiction that $u$ is a cut vertex of $H$. First notice that all vertices which are not in the same connected component as $x$ in $H^{-u}$ are neighbours of $u$. More specifically:
	
	\begin{claim}\label{claim_NonCut3}
	    Let $y$ be a vertex that is not in the same connected component as $x$ in $H^{-u}$. Then $\dist_H(x,y)=\dist_H(x,u)$, and there exists an edge $e_y \supseteq \{u,y\}$ such that, for all shortest $xy$-chain in $H$, $e_y$ is the only edge of the chain containing $u$.
	\end{claim}
	\begin{proofclaim}[Proof of Claim \ref{claim_NonCut3}]
		Let $y$ be a vertex that is not in the same connected component as $x$ in $H^{-u}$: in other words, all $xy$-chains in $H$ contain $u$.
		Let $P$ be a shortest $xy$-chain in $H$, and let $e_y$ be its edge containing $y$. By definition of $u$, we have $\dist_H(x,y)\leq \dist_H(x,u)$, so $u$ cannot appear in any edge of $P$ other than $e_y$, otherwise we would get an $xu$-chain shorter than $P$ (see Figure \ref{fig:lemma_new}). By linearity of $H$, that edge $e_y$ must belong to all shortest $xy$-chains in $H$.
	\end{proofclaim}
	
 	\begin{figure}[h]
 		\centering
 		\includegraphics[scale=.5]{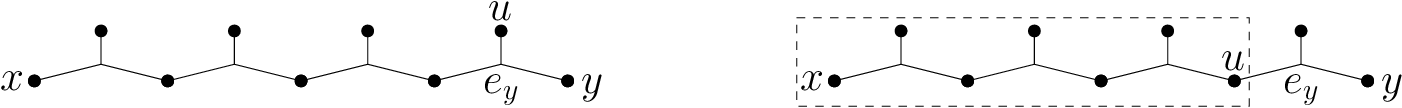}
 		\caption{Left: the actual position of $u$ inside $P$. Right: any other position for $u$ yields an $xu$-chain shorter than $P$.}\label{fig:lemma_new}
 	\end{figure}
	
	Let $y$ be a vertex that is not in the same connected component as $x$ in $H^{-u}$, which exists since $u$ is a cut vertex of $H$. Let $e_y$ be given by Claim~\ref{claim_NonCut3}. Note that, by linearity of $H$, all neighbours of $y$ not in $e_y$ are in the same connected component as $y$ in $H^{-u}$, so we can also apply Claim~\ref{claim_NonCut3} to them.
	
	\begin{claim}\label{claim_NonCut4}
		There cannot exist an edge $\{u,z_1,z_2\}$ where $z_1,z_2 \in N(y) \setminus e_y$.
	\end{claim}
	\begin{proofclaim}[Proof of Claim \ref{claim_NonCut4}]
	
		Suppose for a contradiction that such an edge $e=\{u,z_1,z_2\}$ exists, and let $P$ be a shortest $xz_1$-chain in $H$. We deduce from Claim~\ref{claim_NonCut3} that the only edge of $P$ containing $u$ is the one containing $z_1$. By linearity of $H$, that edge is necessarily $e$, so we get an $xz_2$-chain shorter than $P$ in $H$, contradicting the fact that $\dist_H(x,z_2)=\dist_H(x,u)=\dist_H(x,z_1)$ by Claim~\ref{claim_NonCut3}.
	\end{proofclaim}
	
	\noindent Putting Claims \ref{claim_NonCut3} and \ref{claim_NonCut4} together, we see that all elements of $N(y) \setminus e_y$ are neighbours of $u$, but no two of them can be in the same edge alongside $u$. Therefore:
	$$ d(u) \geq |N(y) \setminus e_{\ell}|+|\{e_{\ell}\}|=|N(y)|-1=2d(y)-1.$$
	However, since $\dist_H(x,y)=\dist_H(x,u)$, we must have $d(u) \leq d(y)$ according to our choice of $u$. The only possibility is that $d(u)=d(y)=1$, which means that $e_y$ is a leaf-edge, a contradiction to the fact that $H$ is reduced.
\end{proof}

\begin{proof}[Proof of Lemma \ref{lem:NonCut1}]
	It suffices to consider any vertex $x$, apply Lemma \ref{lem:NonCut2} on $x$ to get a non-cut vertex $u$, and then apply Lemma \ref{lem:NonCut2} again on $u$ to get a second one.
\end{proof}

\begin{corollary}\label{coro:3unif_A_odd}
	Let $H$ be an odd connected reduced linear 3-uniform hypergraph. Avoider wins on $H$ as last player if and only if $H$ is an isolated vertex. Without the ``reduced'' assumption, the same holds when replacing ``isolated vertex" by ``linear tree''.
\end{corollary}

\begin{proof}
	Suppose $H$ is not an isolated vertex. Since $H$ is odd, Avoider plays first and picks some vertex $x$. Enforcer then picks a non-cut vertex $y \neq x$, which exists by Lemma \ref{lem:NonCut1}, and then Avoider picks some vertex $z$. Since $H^{-y}$ is connected, there exists an $xz$-chain in $H$ that does not contain $y$. In the updated hypergraph, that chain is now a nunchaku. Since a nunchaku has outcome $\SL$, Enforcer wins as second-to-last player by Proposition \ref{prop:monotonicity} (Monotonicity Principle). As for the last assertion of this corollary, it follows from Observation \ref{obs:hypreduced}.
\end{proof}


\subsubsection{Even case: Avoider is second and last}

\noindent Again, we can first address the case where there exists a 2-edge:

\begin{proposition}\label{prop:3unif_A_even}
	Let $H$ be an even connected reduced linear hypergraph of rank 3 which contains a 2-edge. Avoider wins on $H$ as last player if and only if $H$ is an isolated 2-edge.
\end{proposition}

\begin{proof}
	The ``if" direction is trivial, so let us show the ``only if" direction. Suppose Avoider wins on $H$ as last player. In particular, $H$ contains no 0-edge, no 1-edge, and only one 2-edge $e=\{a,b\}$, otherwise we would get a nunchaku in $H$. The idea is to transform into the odd 3-uniform case which we have already addressed. For this, define the hypergraph $H_0$ obtained from $H$ by adding a new isolated vertex $v$: by Corollary \ref{coro:isolated}, the outcome is unchanged. Then, define the hypergraph $H_1$ obtained from $H_0$ by replacing the 2-edge $e=\{a,b\}$ with the superset 3-edge $e_1=\{a,b,v\}$: Avoider obviously wins on $H_1$ as last player, with the same strategy as on $H_0$. Note that $H_1$ is an odd connected linear 3-uniform hypergraph. Therefore, by Corollary \ref{coro:3unif_A_odd}, $H_1$ is a linear tree, so $H$ is a linear tree as well. Since $H$ is reduced, $e$ is the only edge of $H$, which concludes.
\end{proof}

\noindent In the 3-uniform case, we are once again going to need a structural lemma on hypergraphs which is not game-related.

\begin{lemma}\label{lem:cycleprism}
	Let $H$ be a connected reduced linear 3-uniform hypergraph which is neither an isolated vertex, a cycle nor a prism. Then there exists a vertex $y \in V(H)$ such that no connected component of $H^{-y}$ is a linear tree.
\end{lemma}

\begin{proof}
Let $H$ be a connected reduced linear 3-uniform hypergraph  which is neither an isolated vertex, a cycle nor a prism.
Let $y$ be a vertex of minimum degree in $H$. If no connected component of $H^{-y}$ is a linear tree, then we are done. Therefore, we assume from now that there exists a connected component $H'$ of $H^{-y}$ which is a linear tree.

First suppose that $y$ has degree 1, and let $e=\{y,z_1,z_2\}$ be the unique edge containing $y$. It is impossible that $H'$ is an isolated vertex $t$, otherwise $t\in \{z_1,z_2\}$ and $e$ would be a leaf-edge of $H$, contradicting the fact that $H$ is reduced.
Moreover, $H'$ cannot consist of a single edge $e'$, otherwise exactly one vertex of $e'$ would intersect $e$ and $e'$ would be a leaf-edge of $H$, once again contradicting the fact that $H$ is reduced.
Hence, by Proposition \ref{prop:leafedge}, there are at least two leaf-edges in $H'$. Since each leaf-edge of $H'$ is not a leaf-edge of $H$, $e$ must contain at least one vertex of degree 1 of each leaf-edge of $H'$. Since the vertices of degree 1 of the leaf-edges of $H'$ are pairwise distinct, this means $H'$ has exactly two leaf-edges, which contain $z_1$ and $z_2$ respectively as a vertex of degree 1.
Therefore, $H'$ is a $z_1z_2$-chain by Proposition \ref{prop:leafedge}. 
In particular, since $z_1$ and $z_2$ are in the same connected component of $H^{-y}$, we know that $H^{-y}$ is connected, so $H^{-y}=H'$.
All in all, since $H^{-y}$ is a $z_1z_2$-chain, $H$ is a cycle, which is a contradiction.

Now, suppose that $y$ has degree at least 2. This means that there is no vertex of degree 1 in $H$. In particular, $H'$ cannot be an isolated vertex, since that vertex would be of degree 1 in $H$. Furthermore, all the vertices of degree 1 in $H'$ are of degree 2 in $H$ and must be neighbours of $y$. Note that, by Proposition \ref{prop:leafedge}, $H'$ does have vertices of degree~1, so the minimality of $y$ ensures that $y$ is of degree exactly 2. Write the two edges containing $y$ as $\{y,v_1,v_2\}$ and $\{y,w_1,w_2\}$. Since $H$ is linear, two vertices of degree 1 in the same leaf-edge of $H'$ cannot share an edge containing $y$. This means that $H'$ cannot consist of a single edge, as its three vertices would be of degree 1. Therefore, by Proposition \ref{prop:leafedge}, $H'$ has at least two leaf-edges, which amounts to at least four vertices of degree~1. Since those vertices are neighbours of $y$, we conclude that $H'$ has exactly two leaf-edges (so $H'$ is a chain by Proposition \ref{prop:leafedge}) and exactly four vertices of degree 1 which are $v_1$, $v_2$, $w_1$ and $w_2$ (so the chain $H'$ is of length 2). Let $z$ be the fifth vertex of $H'$: up to exchanging the roles of $w_1$ and $w_2$, we can write the two edges of $H'$ as $\{z,v_1,w_1\}$ and $\{z,v_2,w_2\}$. Since all neighbours of $y$ are in the same connected component of $H^{-y}$, we know that $H^{-y}$ is connected, so $H^{-y}=H'$. All in all, we have $V(H)=\{y,z,v_1,v_2,w_1,w_2\}$ and $E(H)=\{\{y,v_1,v_2\},\{y,w_1,w_2\},\{z,v_1,w_1\},\{z,v_2,w_2\}\}$, so $H$ is a prism, which is a contradiction. 
\end{proof}

\begin{corollary}\label{coro:3unif_A_even}
	Let $H$ be an even connected reduced linear 3-uniform hypergraph. Avoider wins on $H$ as last player if and only if $H$ is either a cycle or a prism.
\end{corollary}

\begin{proof}
	The ``if" direction is given by Proposition \ref{prop:outcome_rank3}. For the ``only if" direction, suppose $H$ is neither a cycle nor a prism. By Lemma \ref{lem:cycleprism}, there exists $x \in V(H)$ such that no connected component of $H^{-y}$ is a linear tree: Enforcer picks a such $y$. We must show that Enforcer wins as second-to-last player on $H^{-y}$. Since $H^{-y}$ is odd, it has an odd connected component $H'$, and we know $H'$ is not a linear tree. Corollary \ref{coro:3unif_A_odd} thus ensures that Enforcer wins on $H'$ as second-to-last player, so Enforcer wins on $H^{-y}$ as second-to-last player by Proposition \ref{prop:monotonicity} (Monotonicity Principle).
\end{proof}

\begin{proof}[Proof of Theorem \ref{theo:3unif_A} and Corollary \ref{coro:3unif_A}]
    All cases for the connected components are covered by Proposition \ref{prop:3unif_A_odd}, Corollary \ref{coro:3unif_A_odd}, Proposition \ref{prop:3unif_A_even} and Corollary \ref{coro:3unif_A_even}, which proves Theorem \ref{theo:3unif_A}. Since computing the reduced hypergraph and testing whether each connected component is a cycle, a prism, an isolated 2-edge or an isolated vertex can clearly be done in polynomial time, Corollary~\ref{coro:3unif_A} ensues.
\end{proof}

\section*{Conclusion and perspectives}\strut
\indent We have completely solved the Avoider-Enforcer game on hypergraphs of rank 2 in all cases, and on linear hypergraphs of rank 3 when Avoider is the last player.

A first line of research would be to solve the case where Avoider is the last player on general hypergraphs of rank 3, without the linearity constraint. The difference with the linear case is that a non-linear path whose endpoints are picked by Avoider is not always winning for Enforcer. In particular, if it has an even total number of vertices, Enforcer will be forced to play in it first, unlike nunchakus. Furthermore, in contrast to the Maker-Breaker convention where all is decided after a constant number of moves \cite{MBrank3}, Avoider and Enforcer may stall by picking ``useless” vertices before a losing structure is created by Avoider. This makes the analysis much harder. Another perspective, still with Avoider playing last, is to look at some linear hypergraphs of higher rank. In particular, the case of linear trees seems accessible using the leaf-edges reduction.

When Enforcer is the last player, it seems difficult to find general results, even for linear hypergraphs of rank 3. Indeed, as discussed in Section \ref{sec:lastmove}, one cannot deduce directly from Theorem \ref{theo:3unif_A} a complete characterisation when Enforcer plays last. Since there are only fews hypergraphs where Avoider wins as last player, Proposition \ref{prop:lastmove} can only be used in very specific circumstances, leaving most cases open. However, when the minimum degree is high enough, the situation often becomes very favorable to Enforcer. Indeed, a consequence of Theorem \ref{theo:3unif_A} is that Avoider only wins playing last if there are vertices of degree 1 or 2. We think that a similar result is true when Enforcer is the last player.

\begin{figure}[h]
    \centering
    \begin{tikzpicture}[scale = 0.7]
        \node[noeud] (a1) at (0,5){};
        \node[noeud] (b1) at (0,4){};
        \node[noeud] (c1) at (0,3){};
        \node[noeud] (d1) at (0,2){};
        \node[noeud] (e1) at (0,1){};
        \node[noeud] (f1) at (0,0){};
        
        \node[noeud] (y) at (2,2.5){};
        \node[left] at (y){$x$};
        
        \node[noeud] (a2) at (4,5){};
        \node[noeud] (b2) at (4,4){};
        \node[noeud] (c2) at (4,3){};
        \node[noeud] (d2) at (4,2){};
        \node[noeud] (e2) at (4,1){};
        \node[noeud] (f2) at (4,0){};

        \draw (-1,0) -- (d1);
        \draw (-1,0) -- (e1);
        \draw (-1,0) -- (f1);
        \draw (-1,1.66) -- (b1);
        \draw (-1,1.66) -- (c1);
        \draw (-1,1.66) -- (f1);
        \draw (-1,3.33) -- (a1);
        \draw (-1,3.33) -- (c1);
        \draw (-1,3.33) -- (e1);
        \draw (-1,5) -- (a1);
        \draw (-1,5) -- (b1);
        \draw (-1,5) -- (d1);

        \draw (5,0) -- (d2);
        \draw (5,0) -- (e2);
        \draw (5,0) -- (f2);
        \draw (5,1.66) -- (b2);
        \draw (5,1.66) -- (c2);
        \draw (5,1.66) -- (f2);
        \draw (5,3.33) -- (a2);
        \draw (5,3.33) -- (c2);
        \draw (5,3.33) -- (e2);
        \draw (5,5) -- (a2);
        \draw (5,5) -- (b2);
        \draw (5,5) -- (d2);

        \draw (a1) -- (2,5.5) -- (a2);
        \draw (2,5.5) to[out=300, in=60] (y);
        \draw (b1) -- (2,4.5) -- (b2);
        \draw (2,4.5) to[out=300, in=60] (y);       
        \draw (c1) -- (2,3.5) -- (c2);
        \draw (2,3.5) to[out=300, in=60] (y);
        \draw (d1) -- (2,1.5) -- (d2);
        \draw (2,1.5) to[out=120, in=240] (y);
        \draw (e1) -- (2,0.5) -- (e2);
        \draw (2,0.5) to[out=120, in=240] (y);       
        \draw (f1) -- (2,-0.5) -- (f2);
        \draw (2,-0.5) to[out=120, in=240] (y);
        
    \end{tikzpicture}
    \caption{This hypergraph has minimum degree 3 and a cut-vertex $x$, but Avoider wins as second-to-last player (note that, when removing $x$, both the left and the right subhypergraphs are prisms).}
    \label{fig:counterexampledegree3}
\end{figure}
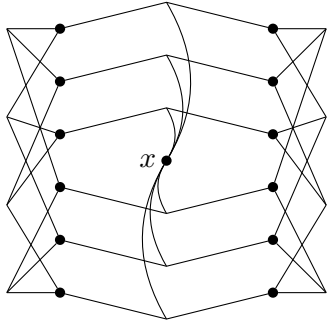

\begin{conjecture}
Let $H$ be a linear hypergraph of rank 3. There exists a constant $d$ such that, if $H$ has minimum degree $d$, then Enforcer wins on $H$ as last player.
\end{conjecture}

This conjecture is true for $d=4$ when $H$ has a cut-vertex $x$. Indeed, under these assumptions, $H^{-x}$ has at least two connected components of minimum degree 3. Since Avoider loses as last player on each of these two components, Enforcer wins on their union as last player, so Enforcer wins on $H$ as last player by monotonicity. Note that the conjecture is not true for $d=3$. Indeed, consider the hypergraph $H$ from Figure \ref{fig:counterexampledegree3}, obtained by taking two disjoint prisms and adding six 3-edges that pairwise connect the vertices of both prisms with a common additional vertex $x$. It has minimum degree 3 and a cut vertex $x$. Since $H^{-x}$ is the disjoint union of two prisms, each of which has outcome $\A$ by Theorem \ref{theo:3unif_A}, Avoider wins on $H^{-x}$ as last player. By Proposition \ref{prop:lastmove}, Avoider thus wins on $H$ as second-to-last player.

Finally, we note that, for linear hypergraphs of rank 3 that are sufficiently connected, two disjoint nunchakus should appear after a small number of moves, thus guaranteeing that Enforcer wins as last player.


\bibliographystyle{abbrv}

\bibliography{biblio}

\end{document}